\documentclass[11pt]{article}
\usepackage{amssymb,amsmath,amsfonts}
\usepackage{tikz}
\usepackage{graphics}
\usepackage{graphicx}

\DeclareMathOperator{\real}{Re}
\DeclareMathOperator{\impart}{Im}

\def\up{\mathbb{C}^+}
\def\eqref#1{$(\ref{#1})$}

\newenvironment{proof}{\noindent {\em {Proof}}.}{$\square$
\medskip}
\newtheorem{theorem}{Theorem}[section]
\newtheorem{corollary}[theorem]{Corollary}
\newtheorem{lemma}[theorem]{Lemma}
\newtheorem{remark}[theorem]{Remark}
\newtheorem{proposition}[theorem]{Proposition}

\newtheorem{example}[theorem]{Example}
\newtheorem{question}[theorem]{Question}
\newtheorem{problem}[theorem]{Problem}

\newcommand{\olz}{\overline z}

\newcommand{\disk}{\mathbb{D}}

\newcommand \ol {\overline}

\newcommand{\bt}[1]{\begin{theorem}\label{#1}}
\newcommand{\bc}[1]{\begin{corollary}\label{#1}}
\newcommand{\bl}[1]{\begin{lemma}\label{#1}}
\newcommand{\bp}[1]{\begin{proposition}\label{#1}}
\newcommand{\be}[1]{\begin{example}\rm\label{#1}}
\newcommand{\bq}[1]{\begin{question}\rm\label{#1}}
\newcommand{\bprob}[1]{\begin{problem}\rm\label{#1}}
\newcommand{\beq}[1]{\begin{eqnarray}\label{#1}}
\newcommand{\br}[1]{\begin{remark}\rm\label{#1}}

\newcommand{\el}{\end{lemma}}
\newcommand{\ep}{\end{proposition}}
\newcommand{\ee}{\end{example}}
\newcommand{\eq}{\end{question}}
\newcommand{\eprob}{\end{problem}}
\newcommand{\eeq}{\end{eqnarray}}
\newcommand{\ed}{\end{definition}}
\newcommand{\et}{\end{theorem}}
\newcommand{\ec}{\end{corollary}}
\newcommand{\er}{\end{remark}}

\title
{Polyanalytic Besov spaces and approximation by dilatations}

\author{Ali Abkar
\vspace{0.5in}\\

Department of Pure Mathematics, Faculty of Science,\\
Imam Khomeini International University, Qazvin 34149, Iran\\
\\
\\
\texttt{Email:~abkar@sci.ikiu.ac.ir}\\}
\date{}
\begin{document}
\maketitle \textbf{Abstract.} Using partial derivatives $\partial_zf$ and $\partial_{\ol z}f$, we introduce Besov spaces of polyanalytic functions on the unit disk and on the upper half-plane. We then prove that the dilatations of each function in polyanalytic Besov spaces converge to the same function in norm. This opens the way for the norm approximation of functions in polyanalytic Besov spaces by polyanalytic polynomials.\\



\textbf{}{}

\section{Introduction}
Let $\Omega$ be a domain in the complex plane. A function $f:\Omega\to \mathbb{C}$ is called $q-analytic$, for $q\in\mathbb{N}$, if $f\in C^q(\Omega)$ and
$$\big (\frac{\partial}{\partial \overline z}\big )^qf=\frac{1}{2^q}\Big (\frac{\partial}{\partial x}+i\frac{\partial}{\partial y}\Big )^q f=0,\quad z=x+iy.$$
A $q$-analytic function $f$ is called a polyanalytic function of degree $q$; or briefly a \textit{polyanalytic function}.
It is well-known that any polyanalytic function $f$ can be written as (see \cite{bal}):
\begin{equation}\label{rep1}
f(z)=h_0(z)+\olz h_1(z)+\cdots+\olz^{q-1}h_{q-1}(z),
\end{equation}
where $h_k,\, 0\le k\le q-1$, are called the analytic components of $f$.
Compared to analytic functions, polyanalytic functions might behave in a different manner; for example, they may vanish on a curve without vanishing identically in the whole domain (see \cite{abr}).\par
Let $L^p(\Omega, dA),\, 0< p<\infty$, denote the space of measurable functions on $\Omega$ equipped with the usual $L^p$ norm; here $dA$ is the area measure on $\Omega$. The polynalytic Bergman space $\mathcal{A}^p_q(\Omega)$ is defined as the space of $q$-analytic functions $f$ on $\Omega$ which belong to $L^p(\Omega)$ as well.
The norm of $f\in\mathcal{A}^p_q(\Omega)$ is given by
$$\| f\|_{\mathcal{A}^p_q(\Omega)}=\left (\int_\Omega |f(z)|^p dA(z)\right )^{1/p}<\infty.$$
A {\it weight} is a positive integrable function $w:\Omega\to [0,\infty )$. We say that $f$ belongs to the weighted polyanalytic Bergman space $\mathcal{A}^p_{q,w}(\Omega)$ if
$$\| f\|_{\mathcal{A}^p_{q,w}(\Omega)}=\left (\int_\Omega |f(z)|^p w(z)dA(z)\right )^{1/p}<\infty.$$
Sometimes we use the short term \textit{(weighted) poly-Bergman space} to refer to this space.
\par
In this paper, $\Omega$ is either the unit disk $\disk=\{z\in\mathbb{C}: |z|<1\}$ or the upper half-plane $\mathbb{C}^+$.
Let us begin by the unit disk.
We recall that the weighted Dirichlet space $D^p_w,\, 0<p<\infty$, consists of analytic functions in the unit disk for which
$$\|f\|^p_{D^p_w}=|f(0)|^p+\int_{\disk}|f^\prime(z)|^pw(z)dA(z)$$
is finite. Similarly, the weighted analytic Besov space $B_w^p$ consists of analytic functions $f$ on the unit disk for which the integral
$$\int_{\disk}(1-|z|^2)^{p-2}|f^\prime(z)|^p w(z)dA(z)<\infty.$$
In the following, we present a definition for the polyanalytic Besov space by using partial differential operators $\partial_zf$ and
$\partial_{\ol z}f$ instead of $f^\prime(z)$.
Let $f$ be a $q$-analytic function on the unit disk, and recall the differential operators ($z=x+iy$):
$$\partial_zf(z)=\frac{1}{2}\left(\frac{\partial f}{\partial x}-i\frac{\partial f}{\partial y}\right ),\quad
\partial_{\overline z}f(z)=\frac{1}{2}\left(\frac{\partial f}{\partial x}+i\frac{\partial f}{\partial y}\right ).$$
For a weight function $w$, we say that a $q$-analytic function $f$ belongs to the weighted polyanalytic Dirichlet space $\mathcal{D}^p_{q,w}$ if
$$\int_{\disk}\Big [|\partial_z f(z)|^p+|\partial_{\overline z} f(z)|^p\Big ]w(z)dA(z)<\infty.$$
The norm of a function in $\mathcal{D}^p_{q,w}$ is given by
$$\|f\|_{\mathcal{D}^p_{q,w}}=\left( |f(0)|^p+\int_{\disk}\Big [|\partial_zf(z)|^p+|\partial_{\overline z} f(z)|^p\Big ] w(z)dA(z)\right )^{1/p}.$$
When $p=2,\, \mathcal{D}^2_{q,w}$ is a Hilbert space of polyanalytic functions whose inner product is
$$\langle f,g\rangle_{\mathcal{D}^2_{q,w}}=\langle f(0), \overline{g(0)}\rangle +\int_{\disk}\Big [\partial_z f(z)\overline{\partial_z g(z)}
+\partial_{\overline z} f(z)\overline{\partial_{\overline z} g(z)}\Big ]
w(z)dA(z).$$
Similarly, we define the weighted polyanalytic Besov space $\mathcal{B}_{q,w}^p$ as the space of $q$-analytic functions $f$ on $\disk$ for which
$$\int_{\disk}(1-|z|^2)^{p-2}\Big [|\partial_zf(z)|^p+|\partial_{\overline z} f(z)|^p\Big ]w(z)dA(z)<\infty.$$
The norm of a function in $\mathcal{B}_{q,w}^p$ is given by
\begin{equation*}
\| f\|_{\mathcal{B}^p_{q,w}}=\left (|f(0)|^p+\int_{\disk}(1-|z|^2)^{p-2}\Big [|\partial_zf(z)|^p+|\partial_{\overline z} f(z)|^p\Big ]w(z) dA(z)\right )^{1/p}.
\end{equation*}
In case that $\Omega$ equals the upper half-plane $\mathbb{C}^+=\{z\in\mathbb{C}: \impart (z)>0\}$, the definition of polyanalytic Besov space on the upper half-plane is slightly different; just replace $f(0)$ by $f(i)$ and $dA(z)$ by a suitable Gaussian weight function (see \S 3). Note that when $q=1$ (which means that $f$ is analytic), we have $\partial_zf(z)=f^\prime(z)$, and $\partial_{\overline{z}}f(z)=0$,
so that the polyanalytic Besov (resp. Dirichlet) space reduces to the analytic Besov (resp. Dirichlet) space. Therefore, the spaces defined above are natural generalizations of the classical Dirichlet and Besov spaces.
We shall at times refer to a polyanalytic Besov space as a \textit{poly-Besov space}; likewise, a polyanalytic Dirichlet space is referred to as
\textit{poly-Dirichlet space}.
\par
It is a natural question in operator theory of function spaces to decide whether the polynomials are dense in the space or not. To answer this question, the most important strategy is to see if the dilatations
$$f_r(z)=f(rz),\quad z\in\Omega,\, 0<r<1,$$
are convergent to $f$. If this holds true, and if each $f_r$ can be approximated by the polynomials, then we are done.
We should point out that for weighted spaces with radial weights (weights that depend only to $|z|$) the problem is rather well-known.
In this paper we consider non-radial weights $w:\Omega\to (0,\infty)$ that satisfy the following condition:\\
There is a constant $C>0$, a non-negative integer $k$, and an $r_0\in (0,1)$ such that
\begin{equation}\label{condition1}
r^kw\left (\frac{z}{r}\right )\le Cw(z),\quad |z|<r,\, r_0\le r<1.
\end{equation}
Under this condition, we prove that the dilatations $f_r$
converge to $f$ in the norm of the weighted polyanalytic Besov space, as $r\to 1^-$.
If, moreover, we assume that the weight $w$ is chosen in such a way that the polyanalytic polynomials (polynomials in $z$ and $\ol z$) are included in the weighted Besov space, and if each $f_r$ can be approximated by the polyanalytic polynomials, then we are able to approximate each $f$ in the weighted Besov
space by polynomials of the form
$$p_{k,m}(z,\olz)=\sum_{i=0}^k\sum_{j=0}^m c_{i,j}z^i\olz^j,\quad c_{i,j}\in \mathbb{C}.$$

 The theory of polyanalytic functions has applications in signal analysis, and in Gabor frames. In particular, Hilbert spaces of polyanalytic functions were extensively used to model physical and engineering problems. The Bergman spaces of polyanalytic functions and the estimation of its reproducing kernels were recently studied by Haimi and Hedenmalm (\cite {hai1}, \cite{hai2}). A good account of polyanalytic Fock spaces of entire functions can be found in the expository article by L.D. Abreu and H.G. Feichtinger \cite{abr}. The subject has also attracted interest within operator theory of function spaces (see \cite{vas1}, \cite{vas2}).

\section{Polyanalytic Besov spaces in the unit disk}
We begin by approximation in polyanalytic Dirichlet spaces.
\begin{theorem}\label{limsup-dirichlet}
Let $0<p<\infty$ and let the weight $w$ satisfies the condition (\ref{condition1}).
Then the polyanalytic polynomials are dense in the weighted polyanalytic Dirichlet space $\mathcal{D}_{q,w}^p$.
\end{theorem}
\begin{proof}
It is enough to show that $f_r\to f$ in the norm of $\mathcal{D}_{q,w}^p$.
To this end, it suffices to work with the following semi-norm (the constant term does not play any role in approximation):
$$
\|f\|^p_{\mathcal{D}^p_{q,w}}=\int_{\disk}\Big [|\partial_zf(z)|^p + |\partial_{\ol z}f(z)|^p\Big ]  w(z)dA(z).
$$
We also note that
\begin{equation}\label{dir-1}
\|f_r -f\|^p_{\mathcal{D}^p_{q,w}}=\int_{\disk}|\partial_zf_r(z)-\partial_zf(z)|^p w(z)dA(z)+\int_{\disk}|\partial_{\ol z}f_r(z)-\partial_{\ol z}f(z)|^p w(z)dA(z).
\end{equation}
It is easy to see that
$$\partial_zf_r(z)=r\partial_zf(rz),\quad \partial_{\ol z}f_r(z)=r\partial_{\ol z}f(rz).$$
By a change of variable (replace $z$ by $z/r$) and using the assumption on the weight function, we obtain
\begin{align*}\int_\disk |\partial_zf_r(z)|^pw(z)dA(z)&=
r^{p-k-2}\int_{r\disk} |\partial_zf(z)|^p r^kw\left(\frac{z}{r} \right)dA(z)\\
&\le Cr^{p-k-2}\int_{r\disk} |\partial_zf(z)|^p w(z)dA(z).
\end{align*}
Therefore we can apply the dominated convergence theorem to get
$$\limsup_{r\to 1^-}\int_\disk |\partial_zf_r(z)|^pw(z)dA(z)\le \int_\disk |\partial_zf(z)|^pw(z)dA(z).$$
This implies that
$$\lim_{r\to 1^-}\int_{\disk}|\partial_zf_r(z)-\partial_zf(z)|^p w(z)dA(z)=0,$$
which means that the first term in \eqref{dir-1} tends to zero.
Similarly, we may verify that
$$\lim_{r\to 1^-}\int_{\disk}|\partial_{\ol z}f_r(z)-\partial_{\ol z}f(z)|^p w(z)dA(z)=0,$$
from which we obtain
$$\lim_{r\to 1^-}\|f_r -f\|^p_{\mathcal{D}^p_{q,w}}=0.$$
\end{proof}

We now prove a similar statement for polyanalytic Besov spaces.
\begin{theorem}\label{limsup-dirichlet}
Let $2\le p<\infty$ and let the weight $w$ satisfies the condition (\ref{condition1}).
Then the polyanalytic polynomials are dense in the weighted polyanalytic Besov space $\mathcal{B}^p_{q,w}$.
\end{theorem}
\begin{proof}
As in the preceding theorem, it suffices to verify that for each
$f\in \mathcal{B}^p_{q,w}$ we have
\begin{equation}\label{bes-1}
\limsup_{r\to 1^{-}}\int_\mathbb{D}(1-|z|^2)^{p-2}|\partial_zf_r(z)|^pw(z)dA(z)\le \int_\mathbb{D}(1-|z|^2)^{p-2}|\partial_zf(z)|^pw(z)dA(z),
\end{equation}
and
\begin{equation}\label{bes-2}
\limsup_{r\to 1^{-}}\int_\mathbb{D}(1-|z|^2)^{p-2}|\partial_{\ol z}f_r(z)|^pw(z)dA(z)\le \int_\mathbb{D}(1-|z|^2)^{p-2}|\partial_{\ol z}f(z)|^pw(z)dA(z).
\end{equation}
To this end, we work with
$$\|f\|^p_{\mathcal{B}^p_{q,w}}=\int_{\disk}(1-|z|^2)^{p-2}\Big [|\partial_zf(z)|^p +|\partial_{\ol z}f(z)|^p \Big ]w(z)dA(z).$$
Therefore, by a change of variables we have
\begin{align*}\|\partial_zf_r\|^p_{\mathcal{B}^p_{q,w}}&=\int_{\disk}(1-|z|^2)^{p-2}|\partial_zf_r(z)|^pw(z)dA(z)\\
&= r^{p-k-2}\int_{r\disk}\left (\frac{r^2-|z|^2}{r^2}\right )^{p-2}|\partial_zf(z)|^pr^kw\left (\frac{z}{r}\right )dA(z).
\end{align*}
But the function
$$r\mapsto \left (\frac{r^2-|z|^2}{r^2}\right )^{p-2}$$
is increasing in $r$ when $p\ge 2$. Now, we apply the monotone convergence theorem to the last integral,
and then invoke the dominated convergence theorem to obtain
$$\limsup_{r\to 1^{-}}\int_\mathbb{D}(1-|z|^2)^{p-2}|\partial_zf_r(z)|^pw(z)dA(z)\le \int_\mathbb{D}(1-|z|^2)^{p-2}|\partial_zf(z)|^pw(z)dA(z),$$
which proves \eqref{bes-1}. The proof of \eqref{bes-2} is similar.
\end{proof}

\section{Polyanalytic Besov spaces in the upper half-plane}
To get an idea of how to define polyanalytic Besov spaces in the upper half-plane, we start by recalling the definition of analytic Bergman spaces on the upper-half plane. An analytic function $f$ on $\up =\{z=x+iy\in \mathbb{C}:\impart (z)=y>0\}$
belongs to the Bergman space if
$$\int_{\up}|f(z)|^p \impart(z)^\alpha e^{-\beta |z|^2}dA(z)<\infty,\quad \alpha\ge 0,\, \beta\ge 0.$$
In the literature, one usually encounters this definition when $\beta=0$ (see \cite{dur}). Therefore, it is natural to define
the weighted poly-Bergman space of the upper half-plane, denoted by $\mathcal{A}^p_{q,w}(\up)$, as the space of $q$-analytic functions on $\up$
for which
$$\|f\|^p_{\mathcal{A}^p_{q,w}(\up)}=\int_{\up}|f(z)|^p w(z)\impart(z)^\alpha e^{-\beta |z|^2}dA(z)<\infty.$$
This suggests that the weighted Dirichlet space and the weighted Besov space of analytic functions should be normed, respectively, by
$$\|f\|^p_{D^p_{\up,w}}=|f(i)|^p+\int_{\up}|f^\prime(z)|^p w(z) \impart(z)^\alpha e^{-\beta |z|^2}dA(z),$$
and
\begin{equation*}
\|f\|_{B^p_{\up,w}}^p=|f(i)|^p+\int_{\up}|f^\prime(z)|^p w(z)\impart(z)^{\alpha+p-2} e^{-\beta |z|^2}dA(z).
\end{equation*}
The advantage of this definition is that, as in the unit disk case, when $p=2$ the Besov space reduces to the Dirichlet space.
In analogy with the above, we now declare the weighted Besov space of polyanalytic functions as the space of $q$-analytic functions on $\up$ such that
\begin{equation*}
\|f\|_{\mathcal{B}^p_{\up,w}}^p=|f(i)|^p+\int_{\up}\Big [|\partial_z f(z)|^p+|\partial_{\ol z} f(z)|^p \Big ] w(z)\impart(z)^{\alpha+p-2} e^{-\beta |z|^2}dA(z)
\end{equation*}
is finite. By our definition, when $p=2$, we get to the Dirichlet space of polyanalytic functions on the upper half-plane, i.e., polyanalytic functions on the upper half-plane for which
\begin{equation*}
\|f\|_{\mathcal{D}^p_{\up,w}}^p=|f(i)|^p+\int_{\up}\Big [|\partial_z f(z)|^p+|\partial_{\ol z} f(z)|^p \Big ] w(z)\impart(z)^{\alpha} e^{-\beta |z|^2}dA(z)<\infty.
\end{equation*}
\begin{theorem}\label{limsup-dirichlet}
Let $0<p<\infty$ and let $w$ satisfies the condition (\ref{condition1}).
Then for each $f\in\mathcal{D}^p_{q,w}(\up)$, the dilatations $f_r$ converge to $f$ in norm.
\end{theorem}
\begin{proof}
As for approximation, it is enough to work with the following semi-norm:
$$\|f\|^p_{\mathcal{D}^p_{q,w}(\up)}=\int_{\up}\Big [|\partial_z f(z)|^p+|\partial_{\ol z} f(z)|^p \Big ] w(z)\impart(z)^\alpha e^{-\beta |z|^2}dA(z).$$
We also note that
\begin{align}\label{dir-up-1}
\|f_r -f\|^p_{\mathcal{D}^p_{q,w}(\up)}&=\int_{\up}|\partial_zf_r(z)-\partial_zf(z)|^p w(z)\impart(z)^\alpha e^{-\beta |z|^2}dA(z)\nonumber \\
&+\int_{\up}|\partial_{\ol z}f_r(z)-\partial_{\ol z}f(z)|^p w(z)\impart(z)^\alpha e^{-\beta |z|^2}dA(z).
\end{align}
Making a change of variable, and using the fact that
$$\partial_zf_r(z)=r\partial_zf(rz),$$
we obtain
\begin{align*}\int_{\up} |\partial_zf_r(z)|^pw(z)&\impart(z)^\alpha e^{-\beta |z|^2}dA(z)\\ &=
r^{p-(k+\alpha+2)}\int_{\up} |\partial_zf(z)|^p r^kw\left(\frac{z}{r} \right)\impart(z)^\alpha e^{\frac{-\beta |z|^2}{r^2}}dA(z)\\
&\le C r^{p-(k+\alpha+2)}\int_{\up} |\partial_zf(z)|^p w(z)\impart(z)^\alpha e^{\frac{-\beta |z|^2}{r^2}}dA(z)
\end{align*}
Therefore the dominated convergence theorem applies to ensure
$$\limsup_{r\to 1^-}\int_{\up} |\partial_zf_r(z)|^pw(z)\impart(z)^\alpha e^{-\beta |z|^2}dA(z)\le \int_{\up} |\partial_zf(z)|^pw(z)\impart(z)^\alpha e^{-\beta |z|^2}dA(z).$$
This entails that the first term on the right-hand side of \eqref{dir-up-1} tends to zero:
$$\lim_{r\to 1^-}\int_{\up}| \partial_zf_r(z)-\partial_zf(z)|^p w(z)\impart(z)^\alpha e^{-\beta |z|^2}dA(z)=0.$$
Similarly, one shows that the second term on the right-hand side of \eqref{dir-up-1} tends to zero, from which we obtain
$$\lim_{r\to 1^-}\|f_r -f\|^p_{\mathcal{D}^p_{q,w}(\up)}=0.$$
\end{proof}

For the Bergman space, the problem can be settled more easily.
\begin{corollary}\label{berg-thm}
Let $0<p<\infty$ and let $w$ satisfies the condition (\ref{condition1}).
Then for each function $f\in\mathcal{A}^p_{q,w}(\up)$, the dilatations $f_r$ converge to $f$ in norm.
\end{corollary}
\begin{proof}
Recalling the definition of norm in $\mathcal{A}^p_{q,w}(\up)$, we just note that
\begin{align*}
\|f_r\|^p_{\mathcal{A}^p_{q,w}{(\up)}}&=\int_{\mathbb{C}^+}|f(rz)|^p w(z)\impart(z)^\alpha e^{-\beta |z|^2}dA(z)\\
&=\frac{1}{r^{\alpha+k+2}}\int_{\mathbb{C}^+}|f(z)|^p r^kw\left (\frac{z}{r}\right )\impart(z)^\alpha e^{\frac{-\beta |z|^2}{r^2}}dA(z)\\
&=\frac{1}{r^{\alpha+k+2}}\int_{\mathbb{C}^+}|f(z)|^p r^kw\left (\frac{z}{r}\right )\impart(z)^\alpha e^{-\beta |z|^2}e^{\beta |z|^2(1-r^{-2})}dA(z).
\end{align*}
Recall that for $r_0<r<1$ we have
$$e^{\beta |z|^2 (1-r^{-2})}\le 1,$$
and hence
\begin{align*}
\|f_r\|^p_{\mathcal{A}^p_{q,w}{(\up)}}&=
\frac{1}{r^{\alpha+k+2}}\int_{\mathbb{C}^+}|f(z)|^p r^kw\left (\frac{z}{r}\right )\impart(z)^\alpha e^{-\beta |z|^2}e^{\beta |z|^2(1-r^{-2})}dA(z)\\
&\le \frac{C}{r^{\alpha+k+2}}\int_{\mathbb{C}^+}|f(z)|^p w(z)\impart(z)^\alpha e^{-\beta |z|^2}dA(z)<\infty.
\end{align*}
This means that the dominated convergence theorem can be applied; so that
$$\lim_{r\to 1^-}\|f_r\|^p_{\mathcal{A}^p_{q,w}{(\up)}}= \int_{\mathbb{C}^+}|f(z)|^p w(z)\impart(z)^\alpha e^{-\beta |z|^2}dA(z)=\|f\|^p_{\mathcal{A}^p_{q,w}{(\up)}},$$
and finally $f_r\to f$ in $\mathcal{A}^p_{q,w}{(\up)}$.
\end{proof}

We now state a similar approximation theorem for the polyanalytic Besov spaces on $\up$.
\begin{theorem}\label{limsup-dirichlet}
Let $2\le p<\infty$ and let $w$ satisfies the condition (\ref{condition1}).
Then for each $f\in\mathcal{B}^p_{q,w}(\up)$, the dilatations $f_r$ are convergent to $f$ in the norm.
\end{theorem}
\begin{proof} Let
$f\in \mathcal{B}^p_{q,w}(\up)$, and let $d\mu(z)=w(z)\impart(z)^{\alpha+p-2} e^{-\beta |z|^2}dA(z)$. We recall that
$$\|f\|^p_{\mathcal{B}^p_{q,w}(\up)}=\int_{\up}\Big [|\partial_zf(z)|^p+|\partial_{\ol z}f(z)|^p\Big ] d\mu(z).$$
Therefore, by a change of variable we have
\begin{align*}\int_{\up}|\partial_zf_r(z)|^p d\mu(z)&=
\frac{r^p}{r^{k+\alpha+p}}\int_{\up}|\partial_zf(z)|^pr^kw\left (\frac{z}{r}\right )\impart(z)^{\alpha+p-2} e^{\frac{-\beta |z|^2}{r^2}}dA(z)\\
&\le \frac{C}{r^{k+\alpha}}\int_{\up}|\partial_zf(z)|^pd\mu (z).
\end{align*}
Now, we apply the dominated convergence theorem to get
$$\limsup_{r\to 1^{-}}\int_{\up}|\partial_zf_r(z)|^pd\mu(z)\le \int_{\up}|\partial_zf(z)|^pd\mu(z).$$
The last inequality shows that
$$\lim_{r\to 1^-}\int_{\up}| \partial_zf_r(z)-\partial_zf(z)|^p w(z)\impart(z)^{\alpha+p-2} e^{-\beta |z|^2}dA(z)=0.$$
The proof that
$$\lim_{r\to 1^-}\int_{\up}| \partial_{\ol z}f_r(z)-\partial_{\ol z}f(z)|^p w(z)\impart(z)^{\alpha+p-2} e^{-\beta |z|^2}dA(z)=0$$
is similar. These equalities show that
$$\lim_{r\to 1^-}\|f_r -f\|^p_{\mathcal{B}^p_{q,w}(\up)}=0.$$
\end{proof}

\begin{example} (a).
Let $\beta$ be a positive number, and $n$ be a positive integer. Then $w(z)=e^{-\beta |z|^n}$ satisfies the condition (\ref{condition1}).
Indeed, for each $0<r<1$, we have
$$w\left (\frac{z}{r}\right )=e^{\frac{-\beta |z|^n}{r^n}}\le e^{-\beta |z|^n}=w(z).$$
(b). Consider the non-radial weight $$w(z)=e^{-\beta |\real(z)|^n}= e^{-\beta |x|^n}.$$
Again, we have $w(z/r)\le w(z)$.\\
(c). In some instances the function $r\mapsto w(z/r)$ may not satisfy the condition (\ref{main-condition}), however there might exist some positive integer $k$ for which the function $r\mapsto r^kw(z/r)$ satisfies the required condition. For instance if $w(z)=\exp(|z|)$, then
$$\frac{d}{dr}w\left (\frac{z}{r}\right )=-\frac{|z|}{r^2}\exp\left (\frac{|z|}{r}\right )<0$$
while
$$\frac{d}{dr}\left[rw(\frac{z}{r})\right]=\left (1-\frac{|z|}{r}\right )\exp\left (\frac{|z|}{r}\right )>0,\quad |z|<r.$$
Note also that $r^2e^{|z|/r}$ is increasing for $r>1/2$ since
$$\frac{d}{dr}(r^2e^{|z|/r})=2re^{|z|/r}+r^2(\frac{-1}{r^2}e^{|z|/r})=(2r-1)e^{|z|/r}>0.$$
\end{example}

\section{Angular weights}

In contrast to radial weights, let us assume that the weight function depends only on the argument of $z$; that is,
$w(re^{i\theta})=w(\theta)$. We may call such weights {\it angular weights}. For example,
$$w(z)=w(re^{i\theta})=(4\pi^2-\theta^2)^\alpha,\quad 0\le \theta <2\pi,\,\,\alpha>0,$$
is an angular weight in the unit disk.
It seems that the study of angular weights was overlooked in the literature. Here we provide some statements on the approximation by polyanalytic polynomials in such weighted spaces. Let us fix the following notations.
\begin{theorem}\label{angular-main}
Let $w:\disk \to(0,\infty)$ be an angular weight satisfying
$$\int_0^{2\pi}w(\theta)d\theta <\infty.$$
Then the $q$-analytic polynomials are dense in $\mathcal{B}^p_{q,w},\, 2\le p<\infty$.
\end{theorem}
\begin{proof}
We first note that for $z=re^{i\theta}$ we have $w(z)=w(z/r)=w(\theta)$, so that by neglecting the constant term in the definition of norm we have
\begin{align*}
\| f_r\|^p_{\mathcal{B}^p_{q,w}}&=\int_{\disk}(1-|z|^2)^{p-2}\Big [|\partial_zf_r(z)|^p+|\partial_{\overline z} f_r(z)|^p\Big ]w(z) dA(z)\\
&=r^{p-2}\int_{r\disk}\left (\frac{r^2-|z|^2}{r^2}\right )^{p-2}\Big [|\partial_zf(z)|^p+|\partial_{\overline z} f(z)|^p\Big ] w(z)dA(z).
\end{align*}
This implies that
$$\limsup_{r\to 1^-} \| f_r\|^p_{\mathcal{B}^p_{q,w}}=\| f\|^p_{\mathcal{B}^p_{q,w}},$$
from which it follows that $f_r\to f$ in norm. But on the unit disk, each $f_r$ can be approximated by $q$-analytic polynomials (this is a consequence of \eqref{rep1}), from which the assertion follows.
\end{proof}

We now consider weights that are multiples of a radial weight and an angular weight.
\begin{theorem}\label{radial-angular-main}
Let $w(se^{i\theta})=\omega(s)v(\theta)$ be a weight function on $\disk$ where $\omega$ and $v$ satisfy
$$\int_0^{1}s\omega(s)ds <\infty,\,\, \int_0^{2\pi}v(\theta)d\theta <\infty,$$
and $r^k\omega(s/r)\le C\omega(s)$ for some integer $k\ge 0$.
Then the $q$-analytic polynomials are dense in $\mathcal{B}^p_{q,w},\, 2\le p<\infty$.
\end{theorem}
\begin{proof}
Again, we see that
\begin{align*}
\| f_r\|^p_{\mathcal{B}^p_{q,w}}&=\int_{\disk}(1-|z|^2)^{p-2}\Big [|\partial_zf_r(z)|^p+|\partial_{\overline z} f_r(z)|^p\Big ]w(s)v(\theta)dA(z)\\
&=r^{p-2-k}\int_{r\disk}\left (\frac{r^2-|z|^2}{r^2}\right )^{p-2}\Big [|\partial_zf(z)|^p+|\partial_{\overline z} f(z)|^p\Big ] r^k w(s/r)v(\theta)dA(z)\\
&\le C r^{p-2-k}\int_{r\disk}\left (\frac{r^2-|z|^2}{r^2}\right )^{p-2}\Big [|\partial_zf(z)|^p+|\partial_{\overline z} f(z)|^p\Big ]w(s)v(\theta)dA(z).
\end{align*}
Therefore, the dominated convergence theorem applies;
$$\limsup_{r\to 1^-}\| f_r\|^p_{\mathcal{B}^p_{q,w}}=\| f\|^p_{\mathcal{B}^p_{q,w}},$$
from which the result follows.
\end{proof}

The above two theorems have upper half-plane analogs as well.
\begin{theorem}\label{dirich-upper-ang}
Let $0<p<\infty$ and let $w$ be an angular weight function on $\up$ satisfying
$$\int_0^{\pi}w(\theta)d\theta <\infty.$$
Then for each $f\in\mathcal{D}^p_{q,w}(\up)$, the dilatations $f_r$ converge to $f$ in norm.
\end{theorem}
\begin{proof}
We note that
\begin{align}\label{dir-up-ang-1}
\|f_r -f\|^p_{\mathcal{D}^p_{q,w}(\up)}&=\int_{\up}|\partial_zf_r(z)-\partial_zf(z)|^p w(z)\impart(z)^\alpha e^{-\beta |z|^2}dA(z)\nonumber \\
&+\int_{\up}|\partial_{\ol z}f_r(z)-\partial_{\ol z}f(z)|^p w(z)\impart(z)^\alpha e^{-\beta |z|^2}dA(z).
\end{align}
By replacing $z$ by $z/r$, and using the fact that
$\partial_zf_r(z)=r\partial_zf(rz),$ and $w(z)=w(z/r)$,
we obtain
\begin{align*}\int_{\up} |\partial_zf_r(z)|^pw(z)&\impart(z)^\alpha e^{-\beta |z|^2}dA(z)\\ &=
r^{p-(\alpha+2)}\int_{\up} |\partial_zf(z)|^p w\left(\frac{z}{r} \right)\impart(z)^\alpha e^{\frac{-\beta |z|^2}{r^2}}dA(z)\\
&\le r^{p-(\alpha+2)}\int_{\up} |\partial_zf(z)|^p w(z)\impart(z)^\alpha e^{-\beta |z|^2}dA(z),
\end{align*}
Therefore the dominated convergence theorem applies to ensure
$$\limsup_{r\to 1^-}\int_{\up} |\partial_zf_r(z)|^pw(z)\impart(z)^\alpha e^{-\beta |z|^2}dA(z)= \int_{\up} |\partial_zf(z)|^pw(z)\impart(z)^\alpha e^{-\beta |z|^2}dA(z).$$
This entails that the first term on the right-hand side of \eqref{dir-up-ang-1} tends to zero:
$$\lim_{r\to 1^-}\int_{\up}| \partial_zf_r(z)-\partial_zf(z)|^p w(z)\impart(z)^\alpha e^{-\beta |z|^2}dA(z)=0.$$
Similarly, one shows that the second term on the right-hand side of \eqref{dir-up-ang-1} tends to zero, from which we obtain
$$\lim_{r\to 1^-}\|f_r -f\|^p_{\mathcal{D}^p_{q,w}(\up)}=0.$$
\end{proof}

\begin{theorem}\label{radial-angular-main-up}
Let $w(se^{i\theta})=\omega(s)v(\theta)$ be an angular weight function on $\up$ where $\omega$ and $v$ satisfy
$$\int_0^{1}s\omega(s)ds <\infty,\,\, \int_0^{\pi}v(\theta)d\theta <\infty,$$
and $r^k\omega(s/r)\le C\omega(s)$ for some integer $k\ge 0$. Then for each $f\in \mathcal{B}^p_{q,w}(\mathbb{C^+}),\, 2\le p<\infty$, the dilatations $f_r$ converge to $f$ in norm.
\end{theorem}
\begin{proof}
Again, we see that
\begin{align*}
\| f_r\|^p_{\mathcal{B}^p_{q,w}(\mathbb{C^+})}&=\int_{\disk}(1-|z|^2)^{p-2}\Big [|\partial_zf_r(z)|^p+|\partial_{\overline z} f_r(z)|^p\Big ]w(s)v(\theta)dA(z)\\
&=r^{p-2-k}\int_{r\disk}\left (\frac{r^2-|z|^2}{r^2}\right )^{p-2}\Big [|\partial_zf(z)|^p+|\partial_{\overline z} f(z)|^p\Big ] r^k w(s/r)v(\theta)dA(z)\\
&\le C r^{p-2-k}\int_{r\disk}\left (\frac{r^2-|z|^2}{r^2}\right )^{p-2}\Big [|\partial_zf(z)|^p+|\partial_{\overline z} f(z)|^p\Big ]w(s)v(\theta)dA(z).
\end{align*}
Therefore, the dominated convergence theorem applies;
$$\limsup_{r\to 1^-}\| f_r\|^p_{\mathcal{B}^p_{q,w}(\mathbb{C^+})}\le \| f\|^p_{\mathcal{B}^p_{q,w}(\mathbb{C^+})},$$
from which the result follows.
\end{proof}


\section{Declarations}

\noindent{\bf Ethical approval}\\
Not applicable\\

\noindent{\bf Competing interests}\\
The author declares no competing interests.\\

\noindent{\bf Authors contribution}\\
Not applicable\\

\noindent{\bf Funding}\\
Not applicable\\

\noindent{\bf Availability of data and materials} \\
Data sharing is not applicable to this article as no data sets were generated or analyzed during the current study.

\end{document}